\documentclass{amsart}

\usepackage{amsmath} 
\usepackage{amsfonts} 
\usepackage{amssymb} 
\usepackage{epsfig} 
\usepackage{graphicx} 
\usepackage{multirow} 
\usepackage{verbatim} 
\usepackage{rotating} 
\usepackage[all]{xy} 
\usepackage{appendix}

\newtheorem{theorem}{Theorem}[section]
\newtheorem{lemma}[theorem]{Lemma}

\theoremstyle{definition}
\newtheorem{definition}[theorem]{Definition}

\theoremstyle{remark}
\newtheorem{remark}[theorem]{Remark}
\theoremstyle{notation}

\numberwithin{equation}{section}
\theoremstyle{corollary}
\newtheorem{corollary}[theorem]{Corollary}

\usepackage[bookmarks=false]{hyperref}
\newcommand{\Map}{\mathrm{Map}}
\newcommand{\map}{\mathrm{map}}

\newcommand{\HOM}{\mathbf{HOM}}

\newcommand{\sSet}{\mathbf{sSet}}

\newcommand{\Cat}{\mathbf{Cat}}

\newcommand{\Top}{\mathbf{Top}}
\newcommand{\Ch}{\mathbf{Ch}}
\newcommand{\C}{\mathbf{C}}

\newcommand{\D}{\mathbf{D}}
\newcommand{\A}{\mathbf{A}}

\newcommand{\G}{\mathrm{G}}
\newcommand{\M}{\mathbf{M}}

\newcommand{\catX}{\mathbf{X}_{\downarrow\Delta}}

\newcommand{\N}{\mathrm{N}}

\newcommand{\sC}{\widehat{\mathbf{C}}}
\newcommand{\sM}{\widehat{\mathbf{M}}}

\newcommand{\sG}{\widehat{\mathbf{G}}}
\newcommand{\holim}{\mathrm{holim}}
\newcommand{\hocolim}{\mathrm{hocolim}}

\newcommand{\colim}{\mathrm{colim}}
\newcommand{\Ob}{\mathrm{Ob}}

\begin{document}

\title{Moduli space of fibrations  in the category of simplicial presheaves}

\author{Ilias Amrani}
\address{Department of Mathematics\\ Masaryk University\\ Czech Republic}
\email{ilias.amranifedotov@gmail.com}
\email{amrani@math.muni.cz}
\thanks{Supported by the project CZ.1.07/2.3.00/20.0003
of the Operational Programme Education for Competitiveness of the Ministry
of Education, Youth and Sports of the Czech Republic.
}

\thanks{}

\subjclass[2000]{Primary 55}




\begin{abstract}
We describe the moduli space of extensions in the model category of simplicial presheaves. This article can be seen as a generalization of \textit{Blomgren-Chacholski} results in the case of simplicial sets. Our description of the moduli space of extensions treat the equivariant and the nonequivariant case in the same setting. As a new result, we describe the moduli space of M-bundles over a fixed space X, when M is a simplicial monoid. Moreover, the moduli space of M-bundles is classified by the classifying space of the simplicial submonoid generated by homotopy invertible elements of M. We give a general interpretation of generalized cohomology theories (connective) in terms of classification of principle bundles. We also construct categorical model for the classifying space BG and EG when G is a simplicial (topological) monoid group like.

\end{abstract}

\maketitle
\section*{Results}
We extend the work done in \cite{blomgren} to a more general setting, we construct a unique theory for the equivariant and nonequivaruant framework, and even more. As consequence, we describe the moduli space of $M$-principle bundle, where $M$ is a simplicial monoid or even more generally $M$ is a simplicial presheaf on some simplicial category. Let $\C$ a simplicial category and $X$ a simplicial set, and $F:\C^{op}\rightarrow\sSet$ any simplicial functor (the category of simplicial presheaves is denoted by $\sC$), then the moduli space $\N_{\bullet} \mathbf{Ext}_{\C}(X,F)$ (cf \ref{Ext}) of extensions of $X$ by $F$ is described as follows:\\\\
\textbf{Theorem I:~}\ref{main}, \ref{main2}
 \textit{The moduli space of extensions is classified by the classifing space of homotopy auto-equivalences of 
 $F$ i.e., the map
$$ \N_{\bullet}\mathbf{Ext}_{\C}(F,X)\sim\Map\big(X,\mathrm{B}Aut^{h}(F)\big).$$ \\
}
Here, the notation $\mathrm{B}Aut^{h}(F)$ means the mapping space of auto-equivalences of a fibrant cofibrant replacement of $F\in \sC$. As an application we can classify the principle $M$-bundles \ref{M-bundle}, when $M$ is a \textit{simplicial monoid}.\\\\
\textbf{Theorem II:~}\ref{ap1}
\textit{Let $M$ be a simplicial monoid and $\mathbf{M}$ the simplicial category with one object $\ast$ and $\mathbf{M}(\ast,\ast)=M$, then
$$\N_{\bullet}\mathbf{Ext}_{\mathbf{M}}(M,X)\sim\Map(X,\mathrm{B}M^{\star}),$$
where $M^{\star}$ (cf \ref{invers}) is the simplicial sub monoid of $M$ of homotopy invertible "elements" . 
}\\\\
If $E$ is an infinity loop space, first, we can define the iterated bar construction $\mathrm{B}^{n}E$ (classifying space) is still an infinity loop space and so we can strictify it to a simplicial monoid group like. As a consequence we obtain the following result:\\\\
\textbf{Corollary:~}\ref{ap3}
\textit{Let $X$ and $E$ as before, then $E^{n}(X)=[X,\mathrm{B}^{n}E]$ is in bijection with the homotopy equivalence classes of $\mathrm{B}^{n-1}E$-principle bundles over $X$. In other words:
$[X,\mathrm{B}^{n}E]\simeq \mathbf{Ext}^{0}_{\mathrm{B}^{n-1}E}(\mathrm{B}^{n-1}E,X):=\pi_{0}\N_{\bullet}\mathbf{Ext}_{\mathrm{B}^{n-1}E}(\mathrm{B}^{n-1}E,X).$\\\\
}
The last corollary gives an interpretation of $H^{n}(X,G)$, the cohomology group with coefficient in some abelian group $G$, the elements of the cohomology group are in bijection with the equivalence classes of objects in the category\\
 $\mathbf{Ext}_{\mathrm{K}(n-1,G)}(\mathrm{K}(n-1,G),X)$, where $\mathrm{K}(n,G)$ is the nth Eilenberg-Maclane space of the group $G$. In the last section \ref{section3}, we prove two result, the first one is an explicit description of the mapping space in the model category of simplicial presheaves $\sC$. Let $F$ be a presheaf, we denote by $\sC_{F}^{we}$ the subcategory of presheaves equivalent to $F$ \ref{weak}.  \\\\
\textbf{Theorem III:} \ref{wMX}
\textit{
For any (fibrant-cofibrant) object $F$ in $\sC$, we have an equivalence of simplicial sets $\N_{\bullet}\sC_{F}^{we}\sim\mathrm{B}Aut^{h}(F)$, where $Aut^{h}(F)$ mapping space of autoequivalence of $F$ i.e., $\map_{\sC}(F,F)^{\star}$. }\\\\
We should remark that in \cite{blomgren}, the authors prove a similar result for any model category, but we think that our result in the case of the model category of simplicial presheaves is more conceptual.  
The second result is about constructing three explicit categories, which connect different mapping spaces of different model categories i.e., the mapping space in the model category of simplicial presheves, and the mapping space of model category of simplicial (topological) categories, these results are summarized in the following lemma.\\\\
\textbf{Lemma: }\textit{ Let G, be a topological monoid group like, then there two categories and a forgetful functor $U:\mathbf{EG}\rightarrow \mathbf{BG}$ such that:\\
1- $\N_{\bullet}\mathbf{BG}\sim BG$ \ref{BG},\\
2- $\N_{\bullet}\mathbf{EG}$ is contractible, \ref{EG}\\
3-  and the realization of the nerve of the homotopy fiber of $U$ is equivalent to $G$ \ref{fiber}. 
}

\section{Preliminaries}
The main goal of this paragraph is to define a $\textit{projective model structure}$ and extra structures on the category of enriched functors $[\C^{op},\sSet]:=\sC$ (or $[\C,\sSet]:=\sC$ it depends on the context) where $(\sSet, \times, \Map)$ is the symmetric monoidal closed model category of simplicial sets, and  $\C$ is  a enriched category over $\sSet$.\\
Lets describe in more detail the category $[\C,\sSet]:=\sC$, the objects are topological functors 
$F: \C\rightarrow \sSet$ i.e., 
\begin{itemize}
\item A map $\Ob F:\Ob \C\rightarrow \Ob \sSet$. 
\item For any pair $(x,y)$ of objects in $\C$, we have a compatible (unit and associativity axioms) continues maps of spaces:
$$F_{x,y}: \C(x,y)\rightarrow \Map(Fx,Fy)$$
or equivalently, a compatible continues map $F_{x,y}:\C(x,y)\times Fx\rightarrow Fy. $
\end{itemize}
A morphism in $\sC$ between two functors $F$ and $G$ is a natural transformation $H: F\rightarrow G$ such that for any object $x\in \Ob\C$ the following diagram commutes:
$$
 \xymatrix{
 \C(x,y)\times Fx\ar[r]^-{F_{x,y}}\ar[d]_-{id\times H_{x}} & Fy\ar[d]^-{H_{y}}\\
 \C(x,y)\times Gx \ar[r]_-{G_{x,y}} & Gy
 }
 $$
 The category $\sC$ is a simplicial category in a natural way. For any simplicial set $X$ and any simplicial functor $F:\C\rightarrow\sSet$ we define the tensor product $X\times F$ and the cotensor $F^{X}$ degree by degree i.e., $(X\times F)(c)=X\times F(c)$ and $F^{X}(c)=\Map(X,F(c))$ for any object $c\in \C$.  
 
 \begin{theorem}[projective model structure]\label{projective}\cite{toen2005homotopical}
The category $\sC$ is cofibrantly generated simplicial model category, where $H:F\rightarrow G$ is a weak equivalence (fibration) if and only if for any $x\in\Ob\C$, the map $H_{x}:Fx\rightarrow Gx$ is a weak equivalence (fibration) in $\sSet$. Moreover, the mapping space in $\sC$ is given by 
$$\map_{\sC}(F,G)_{n}=\hom_{\sC}( \Delta^{n}\times F, G).$$
\end{theorem}
\begin{remark}
The same theorem is true if we replace $\sSet$ by $\Top$ (topological categories) or by 
$\Ch_{k}$ (dg-categories) or any other good symmetric monoidal model category. 
\end{remark}
\begin{lemma}[Yoneda]
Let $\C$ be a simplicial category, then functor $i:\C^{op}\rightarrow \sC$, which takes an object $c$ to $\C(-,c)$, is fully faithful in the enriched sense, i.e., $\C(c,d)=\map_{\sC}(i(c),i(d))$ and more generally 
$$\map_{\sC}(i(c), F)=F(c).$$ 
 \end{lemma}

 
\subsection{the mapping space as a moduli space}
The mapping spaces in $\sC$ has a very nice description as the nerve of some category. 
Let $F$ a fibrant object $\sC$ and $c\in \C$, define the category $F_{\downarrow c}$ as follow:
\begin{enumerate}
\item Object are morphisms $f:~i(c)\times\Delta^{n}\rightarrow F$.
\item morphisms are maps $\sigma: \Delta^{n}\rightarrow \Delta^{m}$ such that 
$f_{1}\circ( id_{c}\times\sigma )= f_{2}.$ 

\end{enumerate}
The nerve of the the category $F_{\downarrow c} $ is related to the mapping space in the projective model category $\sC$, more precisely we have the following equivalence

\begin{lemma}\label{mapping}
We have the following isomorphisms and equivalences:
$$\N_{\bullet} (F_{\downarrow c}) \sim \map_{\sC}(i(c),F)\simeq F(c).$$
\end{lemma}
 \begin{proof}
 The category $F_{\downarrow c}$ is equivalent to the category $F(c)_{\downarrow \Delta}$ (for the definition of the last category (cf \ref{Delta} ), and $\N_{\bullet} (F(c)_{\downarrow \Delta})\sim F(c)$. 
 \end{proof}

\section{Moduli space of extensions}
The main idea is the classification if the 
fiber sequences in the category $\sC$. The standard example of such classification was done in the category of  spaces (topological spaces, simplicial stes). Roughly speaking, given two spaces $X$ and $F$, we want to classify all the extension of the form $E\rightarrow X$ such that the homotopy fiber is equivalent to $F$. We say that two such extensions $E$ and $E^{'}$ are equivalent if there is a homotopy (zig-zag) equivalence between them, compatible with the base $X$. A famous result claims that 
the set of extensions of $X$ by $F$  up to equivalence is in bijection with the set of homotopy classes 
$[X,\mathrm{B}Aut^{h}(F)]$, where $Aut^{h}(F)$ is the topological monoid of homotopy equivalences of $F$ and  $\mathrm{B}Aut^{h}(F)$ is the corresponding classifying space. A more general version of this result is explained in \ref{ap2}.\\ 
In the equivariant setting, the classification of \textit{G-principle bundles} is well understood. Up to isomorphism, the \textit{G-principle bundles} over a space $X$ are classified by the set of homotopy classes $[X,\mathrm{B}G]$, where $\mathrm{B}G$ is the classifying space of $G$. Now, if $G$ is a \textit{simplicial monoid}, we formulate a generalization in our theorem \ref{ap1}. \\
Our main goal is to put the both precedent examples in the same framework and consider the classification in the category of simplicial presheaves $\sC$. 
\begin{definition}\label{trivial}
An object $X$ of $\sC$ is trivial if $X:\C^{op}\rightarrow \sSet$ factors as $X:\C^{op}\rightarrow \ast\rightarrow \sSet$, i.e., $X$ is a constant functor. 
\end{definition}
\begin{definition}\label{Ext}
Let $X$ a trivial object of $\sC$ and $F$ be any objects of $\sC$, define a new category $\mathbf{Ext}_{\C}(F,X)$, where objets are maps $E\rightarrow X$ in $\sC$, such that  the homotopy pullback $\holim(\ast\rightarrow X\leftarrow E )$ is equivalent to $F$. Morphisms are commutative diagrams:
$$\xymatrix{
E\ar[rr]^-{f}_{\sim}\ar[dr] & & E^{'}\ar[dl]\\
& X &
}$$
where $f:E\rightarrow E^{'}$ is a weak equivalence in $\sC$.
\end{definition}

\begin{definition}\label{M-bundle}(M-principle bundle)
Let $X$ be a space and $\M$ a simplicial category with one object $\ast$ i.e., it is equivalent to give
a simplicial monoid $M=\M(\ast,\ast)$ of endomorphisms. We call $M$-principle bundle any object
$f:E\rightarrow X$ of  
$\mathbf{Ext}_{\M}(M,X)$ such that $f$ is a fibration of spaces.
\end{definition}

\begin{definition}\label{Delta}
Suppose that $X$ is a trivial object in $\sC$, we denote by $\catX$ the category  where
\begin{itemize}
\item Object are maps $\Delta^{n}\rightarrow X$ in $\sC$. 
\item Morphisms are commutative diagrams in $\sC$ 
$$\xymatrix{
\Delta^{n}\ar[rr]^-{f}\ar[dr] & & \Delta^{m}\ar[dl]\\
& X &
}$$
\end{itemize}
\end{definition}

\subsection{Global correspondence } In this section we define a correspondence between the category of extension $\mathbf{Ext}_{\C}(F,X)$ and the category of functors $[\catX, \sC]$. Suppose $E\rightarrow X$ an object in $\mathbf{Ext}_{\C}(F,X)$, then for any object in $X\downarrow_{ \Delta}$, we associate the limit
$$\lim ( \Delta^{n} \rightarrow X\leftarrow E)$$ 
in the category $\sC$. This functor is denoted by 
$$\mathrm{Loc}: \mathbf{Ext}_{\C}(F,X)\rightarrow [\catX, \sC].$$
Actually, there is a natural transformation between the functor $\mathrm{Loc}$ and the trivial functor $U_{X}$ which associate to each object $\Delta^{n}\rightarrow X$ in $\catX$ the object $\Delta^{n}\in \sC$. The direct consequence is that the functor $\mathrm{Loc}$ is factored as 
$$\mathrm{Loc}: \mathbf{Ext}_{\C}(F,X)\rightarrow [\catX, \sC]_{\downarrow_{U_{X}}}\rightarrow [\catX, \sC].$$
For any object $f: E\rightarrow X$, there is a functorial factorization $f:E\rightarrow \mathrm{R}E\rightarrow X$ where the first map is a trivial cofibration and the second map is a fibration. Now, we can define an endofunctor  
$\mathrm{R}:\mathbf{Ext}_{\C}(F,X)\rightarrow \mathbf{Ext}_{\C}(F,X)$ which associate to $E\rightarrow X$ the object $\mathrm{R}E\rightarrow X$.

\begin{definition}\label{weak}
Let $\M$ a model category and $X\in \M$ any object, define $\M_{X}^{we}$ to be the subcategory of $\M$ whose objects are equivalent to $X$ and morphisms are weak equivalences.
\end{definition}

\begin{corollary}
The functor $$\mathrm{RLoc}=\mathrm{Loc}\circ\mathrm{R}:  \mathbf{Ext}_{\C}(F,X)\rightarrow
 [\catX, \sC]_{\downarrow_{U_{X}}}$$
factors as $$\mathrm{RLoc}:  \mathbf{Ext}_{\C}(F,X)\rightarrow
 [\catX,\sC^{we}_{F} ]_{\downarrow_{U_{X}}}\rightarrow [\catX,\sC]_{\downarrow_{U_{X}}}.$$ 
\end{corollary} 
\begin{proof}
Suppose that $E\rightarrow X$ is an object If  $\mathbf{Ext}_{\C}(F,X)$, then the pullback 
 $\lim(\mathrm{R}E\rightarrow X\leftarrow \Delta^{n})$ is a homotopy pullback since $\sC$ is right proper, thus the fiber is equivalent to $F$ by definition. It follows that $\mathrm{RLoc}$ factors through  $[\catX,\sC^{we}_{F} ]_{\downarrow_{U_{X}}}$
\end{proof}



\subsection{Local correspondence} 
 The inverse map is called the assembly map in  \cite{blomgren}. The idea is quite simple, pick any object in $K\in [\catX,\sC^{we}_{F}]_{\downarrow_{U_{X}}}$, by definition the object $K$ comes with a natural transformation $K\rightarrow U_{X}$, if we denote by $QK\rightarrow K$ the functorial cofibrant replacement of $K$ in the projective model structure  $[\catX,\sC]$, then after taking the colimit we end up with an object in $\mathbf{Ext}_{\C}(F,X)$ which is given by 
 $$ \hocolim_{\catX}~K=\colim~ QK\longrightarrow \colim~K\longrightarrow\colim~ U_{X}= X.$$

\begin{lemma}\label{fubini}
With the same notation as before, the map (which exists by lifting property)
\begin{equation}\label{fubinieq}
\hocolim_{\catX}~[K(c)]\rightarrow[\hocolim_{\catX}~K](c)
\end{equation}
is a weak equivalence for any $c\in\C$. 
\end{lemma}

\begin{proof}
Just to be more precise, the first homotopy colimit is computed in the category of diagrams in $\sSet$ after evaluation in $c$, and the second one is computed in the category of diagrams in $\sC$ and then evaluated at $c\in \C$. In fact, if $QK$ is the cofibrant replacement of 
 $K$ in the projective model category $[\catX,\sC]$, then $QK(c)$ is the cofibrant replacement of $K(c)$
 in the projective model category $[\catX,\sSet]$, since both categories are cofibrantly generated we can check by hand. The generating (trivial) cofibrations in $[\catX,\sC]$ are given by
 $$ \catX( i, )\times\C(-,x)\times X\rightarrow  \catX( i, )\times\C(-,x)\times Y$$ 
 where $X\rightarrow Y$ is a (trivial) cofibration in $\sSet$,  $i\in \catX$ and $x\in\C$. Thus, the evaluation functor $ev_{c}: [\catX,\sC]\rightarrow [\catX,\sSet]$ at any object $c\in\C$ is a left Quillen functor, which means that $QK(c)$ is a cofibrant object in $[\catX,\sSet]$. We conclude that $QK(c)\rightarrow K(c)$ is a cofibrant replacement of $K(c)$ in  $[\catX,\sSet]$ and that the map in \ref{fubinieq} is a weak equivalence for any $c\in\C$ by the universal property of the the homotopy colimit. 
 \end{proof}

\begin{lemma}\label{Global}
We the same notation as before, suppose that $X\in\sC$ is trivial, then the homotopy pullback in $\sC$ of the diagram: 
$$ \holim\big(\hocolim_{\catX}~K\rightarrow X\leftarrow\ast\big)$$
is equivalent to $F$. 
\end{lemma}

\begin{proof}
We have seen in \ref{fubini} that $\hocolim_{\catX}~[K(c)]\rightarrow[\hocolim_{\catX}~K](c)$,
 is a weak equivalence for any $c\in\C$, on the other hand $X(c)=X$ for any $c\in\C$, we deduce that 
 $$\holim \big([\hocolim_{\catX}~K](c)\rightarrow X(c)\leftarrow\ast\big)\sim\holim \big(\hocolim_{\catX}~[K(c)]\rightarrow X(c)\leftarrow\ast\big)$$
 and by \cite{blomgren} we have
  $$\holim \big(\hocolim_{\catX}~K(c)\rightarrow X(c)\leftarrow\ast\big)\sim F(c)$$
  since the (trivial) fibrations in $\sC$ are degree wise (trivial) fibrations and limits are computed also degrewise in $\sC$, we conclude that 
  $$\holim\big(\hocolim_{\catX}~K\rightarrow X\leftarrow\ast\big)\sim F.$$
 
\end{proof}
\begin{corollary} By the precedent lemma \ref{Global}, we have that $\hocolim_{\catX}K\rightarrow X$ is an object of $\mathbf{Ext}_{\C}(F,X)$, thus we define the derived global functor 
$$\mathrm{LGlob}: [\catX,\sC^{we}_{F}]_{\downarrow_{U_{X}}}\rightarrow  \mathbf{Ext}_{\C}(F,X)$$
 as the composition 
$$ \xymatrix{
[\catX,\sC^{we}_{F}]_{\downarrow_{U_{X}}}\ar[r]^-{Q} & [\catX,\sC^{we}_{F}]_{\downarrow_{U_{X}}}\ar[r]^{\colim} &  \mathbf{Ext}_{\C}(F,X). 
}$$
\end{corollary}
\subsection{Relation between the local-global correspondances }
In this section we prove our main theorem. Recall that the nerve functor $\N_{\bullet}: \Cat\rightarrow \sSet$ has the property that 
\begin{equation}\label{equation11}
\N_{\bullet}\Cat(\C,\D)\simeq\Map(\N_{\bullet}\C,\N_{\bullet}\D).
\end{equation}

\begin{lemma}\label{adjoint}
We have a weak equivalence  
$$\N_{\bullet}\mathrm{Forget}~: \N_{\bullet}[\catX,\sC^{we}_{F}]_{\downarrow_{U_{X}}}\rightarrow \N_{\bullet}[\catX,\sC^{we}_{F}]$$
 \end{lemma}
\begin{proof}
The forgetful functor admits a left adjoint 
$$ - \times U_{X}:~[\catX,\sC^{we}_{F}]\rightarrow [\catX,\sC^{we}_{F}]_{\downarrow_{U_{X}}}$$
such that for any $K\in [\catX,\sC^{we}_{F}]$ the map $K\times U_{X}\rightarrow U_{X}$ is the canonical projection on the second factor. Thus, we obtain the desired equivalence of the corresponding nerves. 
\end{proof}

\begin{theorem}\label{main}
The simplicial map $$\N_{\bullet}\mathrm{RLoc}:  \N_{\bullet}\mathbf{Ext}_{\C}(F,X)\rightarrow
 \N_{\bullet} [\catX, \sC^{we}_{F}]_{\downarrow_{U_{X}}} $$ 
is a weak homotopy equivalence. Moreover, the induced map 
$$\N_{\bullet}\mathrm{LGlob}:  \N_{\bullet} [\catX, \sC^{we}_{F}]_{\downarrow_{U_{X}}} \rightarrow \N_{\bullet}\mathbf{Ext}_{\C}(F,X)$$
is a weak homotopy inverse. 
\end{theorem}
\begin{proof}
For any $c\in \C$, and any $K\in[\catX,\sC^{we}_{F}]_{\downarrow_{U_{X}}}\subset [\catX,\sC]_{\downarrow_{U_{X}}}$, we have seen by \ref{fubini} that  $\hocolim_{\catX}~[K(c)]\rightarrow[\hocolim_{\catX}~K](c)$ is a weak equivalence, which means that $(QK)(c)\rightarrow K(c)$ is a cofibrant replacement for any $c\in \C$ and \textit{functorial} in the variable $c$. Now, we are ready to apply the proposition 18.5 of \cite{blomgren} degreewise, thus we construct a zig-zag of natural transformations 
between the identity and $\mathrm{RLoc}\circ\mathrm{LGlob} $ which is described as follows, after evaluation at $c$
$$\xymatrix{
[\mathrm{RLoc}\circ\mathrm{LGlob} ~ (K)](c) & (QK)(c)\ar[l]\ar[r] & K(c).
}
$$
For any object $f: E\rightarrow X$ in $\mathbf{Ext}_{\C}(F,X)$, we take the evaluation at $c$ i.e., $f_{c}: E(c)\rightarrow X(c)=X$. Applying the same proposition 18.5 of \cite{blomgren}, we have a  zig-zag of natural transformations between $id$ and $\mathrm{LGlob}\circ\mathrm{RLoc} $ given by 
$$\xymatrix{
[\mathrm{LGlob}\circ\mathrm{RLoc} ~ (f_{c})] \ar[r] & Rf_{c} & f_{c}\ar[l] 
}
$$
Consequently, the induced maps 
$$\mathrm{RLoc}: ~\N_{\bullet}\mathbf{Ext}_{\C}(F,X)\rightarrow  \N_{\bullet}[\catX,\sC^{we}_{F}]_{\downarrow_{U_{X}}}$$
$$\mathrm{LGlob}: ~\N_{\bullet}[\catX,\sC^{we}_{F}]_{\downarrow_{U_{X}}}\rightarrow  \N_{\bullet}\mathbf{Ext}_{\C}(F,X)$$
are weak equivalence and are weak inverses of each other. 

\end{proof}
\begin{corollary}\label{main2}
There is a weak  equivalence:
$$  \N_{\bullet}\mathbf{Ext}_{\C}(F,X)\sim \Map(X,\mathrm{B}Aut^{h}(F)).$$
\end{corollary}
\begin{proof}
Since  $\N_{\bullet}\catX \sim X$ by \ref{mapping}, and $\N_{\bullet}\sC^{we}_{F}\sim \mathrm{B}Aut^{h}(F)$ by \ref{wMX}, thus, we apply theorem \ref{main}, \ref{adjoint}, \ref{equation11}, and we conclude that $$ \N_{\bullet}\mathbf{Ext}_{\C}(F,X)\sim\N_{\bullet}[\catX, \sC^{we}_{F}]_{\downarrow_{U_{X}}}\sim \N_{\bullet}[\catX, \sC^{we}_{F}]\sim \Map(X,\mathrm{B}Aut^{h}(F)).$$
\end{proof}

\subsection{Applications}
\begin{definition}\label{invers}
Let $G$ be simplicial monoid, we define $G^{\star}$ the simplicial monoid of homotopy invertible elements as the pullback:
$$
 \xymatrix{
G^{\star}\ar[r]\ar[d] & G\ar[d]\\
 [\pi_{0}G]^{\star}\ar[r] & [\pi_{0}G]
 }
 $$
 where $ [\pi_{0}G]^{\star}$ is the set of invertible elements of the discrete monoid $[\pi_{0}G]$.
 \end{definition}
 

\begin{theorem}\label{ap1}
Let $G$ be a \textit{simplicial monoid} and $\mathbf{G}$ the simplicial category with one object $\ast$ and $\mathbf{G}(\ast,\ast)$, be denote by $G$ the representable functor $\mathbf{G}(-,\ast)$, then 
$$\N_{\bullet}\mathbf{Ext}_{\mathbf{G}}(G,X)\sim\Map(X,\mathrm{B}G^{\star}),$$
if $G$ is simplicial monoid group-like then, $\N_{\bullet}\mathbf{Ext}_{\mathbf{G}}(G,X)\sim\Map(X,\mathrm{B}G).$
\end{theorem}
\begin{proof}
We apply corollary \ref{main2} and recall that $\map_{\widehat{\mathbf{G}}}(G,G)^{\star}\sim G^{\star}$ by lemma \ref{mapping}. 
\end{proof}
\begin{corollary}\label{ap2}
In the model category of simplicial sets $\sSet=\widehat{\ast}$,  we have
$$\N_{\bullet}\mathbf{Ext}_{\widehat{\ast}}(F,X)\sim\Map(X,\mathrm{B}Aut^{h}(F)).$$
\end{corollary}
\begin{proof}
It is a direct consequence of the theorem \ref{main} and \ref{main2} if we replace $\C$ by the trivial category $\ast$. 
\end{proof}
An other application is related to the interpretation of connective cohomology theories, suppose that $E$ is an infinity loop space, then for any space $X$ we define the $n$-th cohomology group $E^{n}(X)=[X,B^{n}E]$. We should recall that $B^{n}E$ is still an infinity loop space and it is equivalent to a simplicial monoid group like. So by the main theorem \ref{main}, we deduce the following lemma:

\begin{corollary}\label{ap3}
Let $X$ and $E$ as before, then $E^{n}(X)$ is in bijection with the homotopy equivalence classes of $\mathrm{B}^{n-1}E$-principle bundles over $X$. In other words:
$$[X,\mathrm{B}^{n}E]\simeq \mathbf{Ext}^{0}_{\mathrm{B}^{n-1}E}(\mathrm{B}^{n-1}E,X),$$
where we donote by definition 
 $\pi_{0}\N_{\bullet}\mathbf{Ext}_{\mathrm{B}^{n-1}E}(\mathrm{B}^{n-1}E,X)=\mathbf{Ext}^{0}_{\mathrm{B}^{n-1}E}(\mathrm{B}^{n-1}E,X).$ 
\end{corollary}
\begin{proof}
Recall that any loop space is (zig-zag) equivalent to strict simplicial monoid, thus $\mathrm{B}^{n-1}E$ is equivalent to strict simplicial monoid which is also a group like, we will denote the strict version also by 
$\mathrm{B}^{n-1}E$. Applying theorem \ref{ap1}, we conclude that  
$$\mathbf{Ext}^{0}_{\mathrm{B}^{n-1}E}(\mathrm{B}^{n-1}E,X)\simeq [X, \mathrm{B}(\mathrm{B}^{n-1}E)^{\star}]\simeq[X,\mathrm{B}^{n}E].$$
\end{proof}
In the particular case when $E$ is the an abelian group $G$, then 
$$H^{n}(X,G)\simeq \pi_{0}\N_{\bullet}\mathbf{Ext}_{\mathrm{K}(n-1,G)}(\mathrm{K}(n-1,G),X).$$
 \begin{remark}
The previous theory of extension can be developed in the topological setting up to some restrictions. Moreover, the topological formulation of theorems \ref{main}, \ref{main2}, \ref{ap1}, \ref{ap2} and lemma \ref{ap3} are still true if we replace $\C$ by an enriched category over CW-complexes, such that $X:\C^{op}\rightarrow CW\subset \Top$ is a trivial topological functor with value in the sub category of CW-complexes. If $G=\mathbb{Z}$ and $X$ a topological space, it is well known that the group $H^{2}(X,\mathbb{Z})$ is isomorphic to the set of isomorphisms  classes of $S^{1}$-principle bundles (the standard definition of principle bundles in the context of topological spaces). This result is a particular case of \ref{ap3}. 
\end{remark}

\section{Categorical model of the classifying space}\label{section3}

In this section, we should mention that we will not be very precise about the set theoretic size issues. 
The problem is solved implicitly by the formula \ref{equation}, where \textit{we compare something small with something apparently big}. The mapping space in the model category of topological (resp. simplicial)  categories $\Cat_{\Top}$ \cite{Amrani1} (resp. $\Cat_{\sSet}$ \cite{bergner}) is strongly related to the mapping space in the model category of simplicial presheaves. This fact is not trivial and goes back to the original paper \cite{toen2007}, where To{\"e}n computes the mapping space of the model category of $dg-categories$. For any (small) topological categories $\D$ (cofibrant) and $\C$, the mapping space $\map_{\Cat_{\Top}}(\D,\C)$ is equivalent to the nerve of the weak groupoid of quasi-representable (qr) topological functors  which are cofibrant-fibrant (cf) in the projective model structure \ref{projective} i.e.,  
\begin{equation}\label{equation0}
\map_{\Cat_{\Top}}(\D,\C)\sim\N_{\bullet} w~ \widehat{\D^{op}\times\C}^{qr,cf}.
\end{equation}
This equivalence is due to B. To{\"e}n (in the context of dg-categories) which we adapted for the case of topological and simplicial categories. He constructed the derived internal Hom, $\mathrm{R}\HOM$ in the model category $dg-\Cat$ which is
 $$\mathrm{R}\HOM(\D,\C)\sim \widehat{\D^{op}\otimes\C}^{qr,cf}.$$ 
 Inspired by his construction and playing with adjunction between the derived internal Hom and the derived tensor product we give a full description of the mapping space in $\Cat_{\Top}$ (resp. $\Cat_{\sSet}$),  (cf. formula \ref{equation0}).
In the case where $\D=\ast$, then we obtain a very nice formula (cf \cite{Amrani1},\cite{toen2007})
\begin{equation}\label{equation}
\widetilde{\N}\C^{\star}\sim\map_{\Cat_{\Top}}(\ast,\C)\sim \widetilde{\N}_{\bullet}\widehat{\C}^{qr,cf,\star}\sim\N_{\bullet} w~ \widehat{\C}^{qr,cf}\sim\N_{\bullet} w~ \widehat{\C}^{qr}
\end{equation}
where $\C^{\star}$ is the $\infty$-groupoid associated to $\C$, and $ \widehat{\C}^{qr}$ is the subcategory of functors which are equivalent to a representable functors $\map_{\C}(-,c)$ for some object $c$. The first  equivalence $\widetilde{\N}\C^{\star}\sim\map_{\Cat_{\Top}}(\ast,\C)$ is described in \cite{Amrani1}, the second equivalence comes from the fact that the inclusion of  $\C$ in $\widehat{\C}^{qr,cf}$ is a Dwyer-Kan equivalence. The  equivalence  $\map_{\Cat_{\Top}}(\ast,\C)\sim\N_{\bullet} w~ \widehat{\C}^{qr,cf}$ is a special case of \ref{equation0}. The last equivalence is a consequence of the fact that the cofibrant-fibrant replacement is functorial.  

\begin{lemma}\label{tech1}
Let $\A$ and $\M$ two model categories with functorial fibrant $\mathrm{R}$ and cofibrant $\mathrm{Q}$ replacement, suppose  we have a Quillen adjunction 
$  \xymatrix{G:~\A \ar@<2pt>[r]^{ } & \M~ : F, \ar@<2pt>[l]^{}}$
and let $a\in \A$ and $m\in\M$, such that 
\begin{enumerate}
\item $G\mathrm{Q}(a)\in  \M^{we}_{m}.$
\item $F\mathrm{R}(m)\in \A^{we}_{a}.$
\end{enumerate}
then there is an equivalence of simplicial sets  $\N_{\bullet}\A^{we}_{a}  \sim \N_{\bullet}\M^{we}_{m}.$
\end{lemma}
\begin{proof}
Since the adjunction verify the properties (1) and (2), we can restrict the adjunction to
$$  \xymatrix{G\mathrm{Q}:~ \A^{we}_{a} \ar@<2pt>[r]^{ } & \M^{we}_{m}:~F\mathrm{R}  \ar@<2pt>[l]^{}}.$$ 
Moreover, we have a zig-zag of natural transformations:\\
$$G\mathrm{Q}F\mathrm{R}\rightarrow \mathrm{R}\leftarrow id ~ \textrm{ and } ~  F\mathrm{R}G\mathrm{Q}\leftarrow \mathrm{Q}\rightarrow id.$$
We conclude that $ \N_{\bullet}\A^{we}_{a}  \sim \N_{\bullet}\M^{we}_{m}.$
\end{proof}

\begin{theorem}\label{wMX}
For any (fibrant-cofibrant) object $F$ in $\sC$, we have an equivalence of simplicial sets $\N_{\bullet}\sC_{F}^{we}\sim\mathrm{B}Aut^{h}(F)$ where $Aut^{h}(F)$ mapping space of autoequivalence of $F$ i.e., $\map_{\sC}(F,F)^{\star}$.  
\end{theorem}
\begin{proof}
Let denote the monoid $\map_{\sC}(F,F)$ by $M$, and let $\M$ be the simplicial category with one object and $\ast$ and $\map_{\M}(\ast,\ast)=M$. Let $\sM$ be the simplicial category of simplicial functors $[\M,\sSet]$. We have a natural Quillen adjunction 
 $$  \xymatrix{-\otimes_{M} F:~\sM \ar@<2pt>[r] & \sC:~ \map_{\sC}(F,-)\ar@<2pt>[l] }$$
 Since  $F\otimes_{M}^{\mathrm{L}}M\sim F\otimes _{M}M\simeq F$  and  $\mathrm{R}\map_{\sC}(F,F)\sim\map_{\sC}(F,F)=M$,
 we can apply lemma \ref{tech1} for $a=F$ and $m=M$, thus we conclude that 
 $$ \N_{\bullet}\sC^{we}_{F}\sim \N_{\bullet}\sM_{M}^{we}=\N_{\bullet}w\sM^{qr}$$ 
 is a weak equivalence. On an other hand, applying the formula \ref{equation} we obtain an equivalence  
 $$ \N_{\bullet}w\sM^{qr}\sim \widetilde{\N}_{\bullet} \M^{\star}\sim \mathrm{B}M^{\star}. $$
 But $M^{\star}=\map_{\sC}(F,F)^{\star}$ by definition. 
\end{proof}

\subsection{Classifying Category}
In this section, $G$ will denote a topological monoid group like i.e., $\pi_{0}G$ is a group. The underlying space $G$ is compactly generated and Hausdorff. In this paragraph, we define the classifying category $\mathbf{BG}$ of $G$. We also define in a categorical way the $\mathbf{EG}$ (which has a contractible nerve) and we construct a functor 
$\mathbf{EG}\rightarrow \mathbf{BG}$ in such a way that the realization of  the comma category (homotopy fiber) is equivalent to $G$.

\begin{definition}[$G$-space]\label{G-space}
A $G$-space $X$ is a topological space with a continues action of $G$.
\end{definition}
\begin{remark}
A $G$-space $X$ is the same thing as giving a functor $X:\mathbf{G}\rightarrow \Top$, where $\mathbf{G}$ is a topological category with one object $\ast$ and $\mathbf{G}(\ast,\ast)=G$. The category of $G-spaces$ is denoted by $\sG$. 
\end{remark}
 \begin{definition}[Categorical classifying space]\label{BGdef}
 Let $G$ a topological monoid group like, define the category $\mathbf{BG}$ as follows:
 \begin{itemize}
 \item Objects are cofibrant $G$-spaces equivalent to $G$ (in the projective model structure $\sG$).
 \item Morphisms are $G$-maps $f:X\rightarrow Y$  which are weak equivalence.
  \end{itemize}
\end{definition} 
\begin{remark}
 By definition, the category $\mathbf{BG}$ is exactly the category $w\sG^{qr,fc}$.  
\end{remark}
\begin{lemma}\label{BG}
The nerve of the category $\mathbf{BG}$ is equivalent to the classifying space $BG$,
$$ \N_{\bullet}\mathbf{BG}\sim \mathrm{B}G.$$
\end{lemma}
\begin{proof}
Since $\mathbf{G}$ is an $\infty-groupoid$ by definition, then $\sG^{qr,fc}$ is also an $\infty-groupoid$ equivalence to $\mathbf{G}$. Thus, we apply the formula \ref{equation} and conclude that
  $\widetilde{\N}_{\bullet}\mathbf{G}\sim\N_{\bullet} \sG^{qr,fc}$,
but $\widetilde{\N}_{\bullet}\mathbf{G}$ is a model for $\mathrm{B}G$ and $\mathbf{BG}= w\sG^{qr,fc}=\sG^{qr,fc}$.
\end{proof}
\begin{definition}
The category $\mathbf{EG}$ is defined as follows:
\begin{enumerate}
\item Objects are morphisms of $G$-spaces $X\rightarrow G$  (not necessary weak equivalences), such that $X$ is a cofibrant $G$-space and weakly equivalent to $G$.
\item the morphisms $\mathbf{EG}(X\rightarrow G, Y\rightarrow G)$ are given by the commutative diagram of $G$-spaces.  

$$
 \xymatrix{
 X\ar[dr]_-{}\ar[rr]_-{\sim} & &Y\ar[dl]^-{}\\
  & \G &
 }
 $$
\end{enumerate}
\end{definition}
We should remark that there is an obvious (forgetful functor) from $U:\mathbf{EG}\rightarrow \mathbf{BG}$, which sends $X\rightarrow G$ to $X$. 
\begin{lemma}\label{fiber}
The homotopy fiber of the map $|\N_{\bullet} U|:|\N_{\bullet}\mathbf{EG}|\rightarrow|\N_{\bullet}\mathbf{BG}|$ is equivalent to $ G$. 
\end{lemma}
\begin{proof}
Using the Quillen theorem B, it is enough to study the homotopy type of the comma category $\mathbf{EG}_{\downarrow X}$, where $X\in\mathbf{BG}$. We will show that $\N_{\bullet} \mathbf{EG}_{\downarrow X}$ is equivalent to the mapping space $\map_{\sG}(\G,\G)$ in the projective  model structure \ref{projective}. More precisely, the comma category $\mathbf{EG}_{\downarrow X}$ is described as follows:
\begin{itemize}
\item The objects are maps $h: UY\rightarrow X$ in $\mathbf{BG}$. Where $Y$ is an object of 
$\mathbf{EG}$. 
\item A morphism between  $h: UY\rightarrow X$ and $h': UY'\rightarrow X$, is a morphism $f:Y\rightarrow Y^{'}$ in $\mathbf{EG}$ such that the following diagram commutes:
$$
 \xymatrix{
 UY\ar[rr]^-{Uf}_-{\sim}\ar[dr]_-{h}^-{\sim} & &UY'\ar[dl]^-{h'}_-{\sim}\\
  & X &
 }
 $$
\end{itemize}
It is easy to see that the category $\mathbf{EG}_{\downarrow X}$ is isomorphic to the category where 

\begin{itemize}
\item Objects are zigzag maps in $\mathbf{BG}$ of the form $\xymatrix { G & Z\ar[l]\ar[r]^-{\sim} &X}$.
\item Morphisms are maps $Z\rightarrow Y$ such that the following diagram commutes:
$$
 \xymatrix{
 G& Z\ar[r]^-{\sim}\ar[l]_-{}\ar[d]^-{\sim} & X\\
  & Y\ar[ul]^-{}\ar[ur]_-{\sim} &
 }
 $$
  
\end{itemize}
By \cite{dugger2006}, the nerve of the comma category $\mathbf{EG}_{\downarrow X}$ is equivalent to the mapping space $ \map_{\sG}(X,G)$ in the projective model structure defined in \ref{projective}. On an other hand any weak equivalence $X\rightarrow Y$ of $G$-spaces induces a weak equivalence $\map_{\sG}(Y,G)\rightarrow\map_{\sG}(X,G)$ because $X, ~Y$ are cofibrant and $G$ is fibrant as a $\G-space$. Since there is a weak equivalence $X\rightarrow G$ of $G$-spaces by definition, we conclude that
$$|\map_{\sG}(X,G)|\sim|\map_{\sG}(G,G)|\sim G,$$
for any $X\in \mathbf{BG}$. 
\end{proof}
\begin{corollary}\label{EG}
The nerve of the category $\mathbf{EG}$ is contractible. 
\end{corollary} 
\begin{proof}
It is an easy consequence of \ref{fiber} and \ref{BG} and Serre's long exact sequence in homotopy.
\end{proof}

\bibliographystyle{plain} 
\bibliography{modulifib}

\begin{thebibliography}{1}

\bibitem{amrani2011grothendieck}
I.~Amrani.
\newblock Grothendieck's homotopy hypothesis.
\newblock {\em Arxiv preprint arXiv:1112.1251}, 2011.

\bibitem{Amrani1}
I.~Amrani.
\newblock Model structure on the category of topological categories.
\newblock {\em arxiv.org/pdf/1110.2695}, 2011.

\bibitem{bergner}
J.E. Bergner.
\newblock {A model category structure on the category of simplicial
  categories}.
\newblock {\em Transactions-American Mathematical Society}, 359(5):2043, 2007.

\bibitem{blomgren}
M.~Blomgren and W.~Chacholski.
\newblock On the classification of fibrations.
\newblock {\em arXiv preprint arXiv:1206.4443}, 2012.

\bibitem{dugger2006}
D.~Dugger.
\newblock {Classification spaces of maps in model categories}.
\newblock {\em Arxiv preprint math/0604537}, 2006.

\bibitem{toen2007}
B.~To{\"e}n.
\newblock {The homotopy theory of dg-categories and derived Morita theory}.
\newblock {\em Inventiones mathematicae}, 167(3):615--667, 2007.

\bibitem{toen2005homotopical}
B.~To{\"e}n and G.~Vezzosi.
\newblock Homotopical algebraic geometry 1: Topos theory.
\newblock {\em Advances in Mathematics}, 193(2):257--372, 2005.

\end{thebibliography}

\end{document}